\documentclass[a4paper]{amsart}
\usepackage{amsthm,amsfonts,amsmath,amssymb}
\usepackage[abs]{overpic}
\usepackage{comment} 

\newtheorem{theorem}{Theorem}[section]
\newtheorem{lemma}[theorem]{Lemma}

\newtheorem{proposition}[theorem]{Proposition}
\newtheorem{corollary}[theorem]{Corollary}

\theoremstyle{definition}

\newtheorem{remark}[theorem]{Remark}

\theoremstyle{remark}
\newtheorem*{acknowledgements}{Acknowledgements}

\newcommand{\Z}{\mathbb{Z}}

\makeatletter

\makeatletter

\@addtoreset{table}{section}
\makeatother

\makeatletter
  
  \@addtoreset{equation}{section}
\makeatother

\begin{document}
\title{A note on Fox colorings of virtual tangles}

\author[T. NAKAMURA]{Takuji NAKAMURA}
\address{Faculty of Education, 
University of Yamanashi,
Takeda 4-4-37, Kofu, Yamanashi 400-8510, Japan}
\email{takunakamura@yamanashi.ac.jp}

\author[Y. NAKANISHI]{Yasutaka NAKANISHI}
\address{Department of Mathematics, Kobe University, 
Rokkodai-cho 1-1, Nada-ku, Kobe 657-8501, Japan}
\email{nakanisi@math.kobe-u.ac.jp}

\author[S. SATOH]{Shin SATOH}
\address{Department of Mathematics, Kobe University, 
Rokkodai-cho 1-1, Nada-ku, Kobe 657-8501, Japan}
\email{shin@math.kobe-u.ac.jp}

\author[K. Wada]{Kodai Wada}
\address{Department of Mathematics, Kobe University, Rokkodai-cho 1-1, Nada-ku, Kobe 657-8501, Japan}
\email{wada@math.kobe-u.ac.jp}

\makeatletter
\@namedef{subjclassname@2020}{%
  \textup{2020} Mathematics Subject Classification}
\makeatother


\thanks{This work was supported by JSPS KAKENHI Grant Numbers 
JP20K03621, JP22K03287, and JP23K12973.}



\begin{abstract}
We study Fox colorings of tangle diagrams 
by $R=\mathbb{Z}$ or $\mathbb{Z}/p\mathbb{Z}$, 
where $p\geq3$ is an odd integer. 
For an $R$-colored $m$-string tangle diagram, 
the colors at the $2m$ boundary points 
form a vector $v\in R^{2m}$. 
We show that 
for classical tangle diagrams, 
such vectors are completely characterized 
by the alternating sum condition $\Delta(v)=0$. 
We then investigate how this restriction changes in the virtual setting. 
For $R=\mathbb{Z}$, 
the realizability of $v$ is determined by 
a divisibility condition on $\Delta(v)$. 
For $R=\mathbb{Z}/p\mathbb{Z}$, 
every vector is realizable by a virtual tangle diagram. 
\end{abstract}

\maketitle

\section{Introduction} 

In 1970, Conway \cite{Con} introduced the notion of tangles 
as a method for studying knots through decomposition. 
Tangles provide a language for analyzing how local properties 
of a knot influence its global structure. 
For example, Krebes~\cite{Kre} proved that 
if a tangle $T$ embeds in a knot $K$, 
then any integer dividing the determinants of 
both numerator and denominator closures of $T$ 
also divides the determinant of $K$. 

Fox colorings~\cite{Fox} provide a fundamental tool in knot theory. 
Although their definition is elementary, 
Fox colorings capture significant information about knots. 
For instance, the $\mathbb{Z}/p\mathbb{Z}$-coloring number 
is an effective invariant for distinguishing knots and 
gives a lower bound for the unknotting number~\cite{Prz}.

Fox colorings also play an important role in the study of tangles. 
Silver and Williams~\cite{SW} extended the above result of Krebes 
to virtual tangles and virtual knots 
using $\Z/p\Z$-colorings. 
Furthermore, Kauffman and Lambropoulou~\cite{KL} 
utilized $\Z$-colorings to give a combinatorial proof 
of Conway's classification of rational tangles. 
These results suggest that colorings provide a useful tool 
for understanding structural properties of tangles. 

A notable observation by Przytycki~\cite{Prz} 
states that if a $1$-string classical tangle diagram  
is $\mathbb{Z}/3\Z$-colored, 
then the two endpoints must receive the same color. 
By contrast, Silver and Williams~\cite{SW} demonstrated that 
for any $a,b\in\Z/3\Z$, 
there exists a $\mathbb{Z}/3\Z$-colored $1$-string virtual tangle diagram 
whose endpoints receive the colors $a$ and $b$. 
This phenomenon reveals a striking contrast 
between classical and virtual tangle diagrams. 

Motivated by this contrast, 
we investigate how the restrictions on Fox colorings 
differ between classical and virtual tangle diagrams.
Let $R=\Z$ or $\Z/p\Z$, where $p$ is an odd integer. 
Let $T$ be a tangle diagram 
consisting of $m\geq1$ strings in a $2$-disk, 
and $C$ an $R$-coloring of $T$. 
For an $R$-colored tangle diagram $(T,C)$, 
the colors appearing at the $2m$ boundary points of $T$ 
in cyclic order along the boundary 
form the \emph{boundary color vector}
\[
v=(a_1,\dots,a_{2m})\in R^{2m}, 
\]
which is denoted by $v=\partial(T,C)$. 

For classical tangle diagrams, 
the boundary color vectors are not arbitrary. 
It is known~\cite{Prz} that 
the boundary color vector $v$ satisfies 
\[
\Delta(v)=\sum_{i=1}^{2m}(-1)^{i-1}a_i=0\in R.
\]
See the left-hand side of Figure~\ref{fig:example}. 
In this paper, we show that this condition completely characterizes 
the boundary color vectors arising from classical $R$-colored tangles.

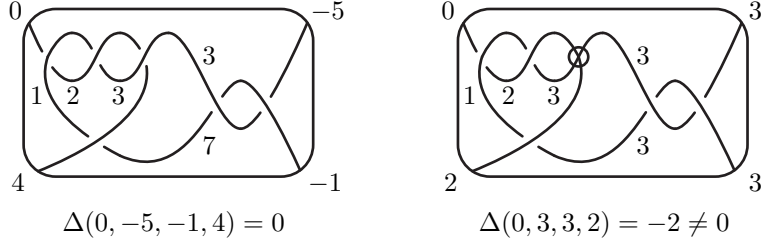
\begin{figure}[htb]
  \centering
    \begin{overpic}[]{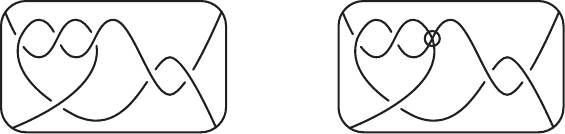}
      \put(15,-20){$\Delta(0,-5,-1,4)=0$}
      \put(171,-20){$\Delta(0,3,3,2)=-2\ne 0$}
      \put(-5,60){$0$}
      \put(108,60){$-5$}
      \put(107,-5){$-1$}
      \put(-4,-5){$4$}
      \put(3,28){$1$}
      \put(16.5,28){$2$}
      \put(33.5,28){$3$}
      \put(67.5,43){$3$}
      \put(67.5,9){$7$}
      \put(157,60){$0$}
      \put(272,60){$3$}
      \put(272,-5){$3$}
      \put(158,-5){$2$}
      \put(165,28){$1$}
      \put(179.5,28){$2$}
      \put(196.5,28){$3$}
      \put(230,43){$3$}
      \put(230,9){$3$}
    \end{overpic}
  \vspace{1.5em}
  \caption{${\Z}$-colored classical and virtual tangle diagrams}
  \label{fig:example}
\end{figure}

\begin{proposition}\label{prop11}
For a vector $v\in R^{2m}$, 
the following are equivalent. 
\begin{enumerate}
\item
There exists an $R$-colored classical tangle diagram $(T,C)$ 
such that $\partial(T,C)=v$. 
\item
$\Delta(v)=0$. 
\end{enumerate}
Moreover, such a diagram $T$ can be chosen to satisfy 
the property $\mathrm{P}_1$ as follows: 
\begin{itemize}
\item[$\mathrm{P}_1$:] 
$T$ can be transformed into a trivial tangle diagram 
by a finite sequence of moves 
shown in {\rm Figure~\ref{fig:interchange}(a)}, 
each of which interchanges adjacent endpoints of strings 
along the boundary of the $2$-disk containing $T$. 
\end{itemize}
\end{proposition}

\begin{figure}[htb]
  \centering
    \begin{overpic}[]{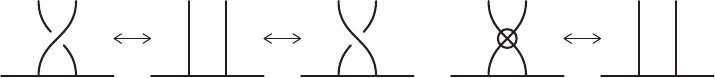}
      \put(94,-13){(a)}
      \put(275,-13){(b)}
    \end{overpic}
  \vspace{1em}
  \caption{Two moves (a) and (b)}
  \label{fig:interchange}
\end{figure}

While the condition that $\Delta(v)=0$ in Proposition~\ref{prop11} 
characterizes the classical case, 
the situation for virtual tangle diagrams is less restrictive. 
See the right-hand side of Figure~\ref{fig:example}. 
The main purpose of this paper is to investigate 
how this restriction changes for virtual tangle diagrams. 
More precisely, we determine exactly which 
boundary color vectors arise from $R$-colored virtual tangles.

We first consider the case $R=\Z$. 
We say that a vector $v=(a_1,\dots,a_{2m})\in\Z^{2m}$ 
is \emph{trivial} if $a_1=\dots=a_{2m}$, 
and \emph{nontrivial} otherwise. 
For a nontrivial $v$, 
we define 
\[d(v)=\gcd\{a_i-a_j\mid 1\leq i\ne j\leq 2m\}>0,\] 
which can be uniquely written as $d(v)=2^{k(v)}(2n+1)$
with $k(v),n\ge 0$. 
Our main theorem states that the realizability of $v$ 
as the boundary color vector of a ${\Z}$-colored virtual tangle diagram is 
determined by $\Delta(v)$ and $k(v)$.

\begin{theorem}\label{thm12}
For a vector $v\in\Z^{2m}$, 
the following are equivalent. 
\begin{enumerate}
\item
There exists a $\Z$-colored virtual tangle diagram $(T,C)$ 
such that $\partial(T,C)=v$. 
\item 
Either 
\begin{enumerate}
\item
$v$ is trivial, or 
\item
$v$ is nontrivial and 
$\Delta(v)$ is divisible by $2^{k(v)+1}$. 
\end{enumerate}
\end{enumerate}
Moreover, if $m\geq 2$, 
such a diagram $T$ can be chosen 
to satisfy the properties $\mathrm{P}_2$ and $\mathrm{P}_3$ 
as follows: 
\begin{itemize}
\item[$\mathrm{P}_2$:] 
$T$ has at most one virtual crossing. 
\item[$\mathrm{P}_3$:] 
$T$ can be transformed into a trivial tangle diagram 
by a finite sequence of moves 
shown in {\rm Figure~\ref{fig:interchange}(a)} and {\rm (b)}, 
each of which interchanges adjacent endpoints of strings 
on the boundary of the $2$-disk containing $T$. 
\end{itemize}
If $m=1$, every boundary color vector $v\in\Z^2$ is trivial. 
\end{theorem}

The case $R=\Z/p\Z$ is quite different. 
In this case, the classical restriction disappears completely. 

\begin{corollary}\label{cor13}
For any vector $v\in(\Z/p\Z)^{2m}$, 
there exists a $\Z/p\Z$-colored virtual tangle diagram $(T,C)$ 
such that $\partial(T,C)=v$. 
Moreover, such a diagram can be chosen 
to satisfy the property $\mathrm{P}_2$ when $m=1$, 
and the properties $\mathrm{P}_2$ and $\mathrm{P}_3$ when $m\geq 2$. 
\end{corollary}

The paper is organized as follows. 
Section~\ref{sec2} reviews $R$-colorings of tangle diagrams 
and introduces four quantities associated with a vector $v$: 
the alternating sum $\Delta(v)$, 
the integers $d(v)$ and $k(v)$, 
and a multiset $M(v)$. 
We also collect several results on 
the Hurwitz action of the braid group on $\Z^{2m}$, 
which will be used in the proofs of 
Proposition~\ref{prop11} and Theorem~\ref{thm12}. 
Section~\ref{sec3} studies $R$-colorings of 
classical tangle diagrams and 
proves Proposition~\ref{prop11}. 
Finally, Section~\ref{sec4} investigates $R$-colorings 
of virtual tangle diagrams and 
proves Theorem~\ref{thm12} and Corollary~\ref{cor13}.

\section{Preliminaries}\label{sec2}

\subsection{The boundary color vector} 

Let $m\geq1$ be an integer. 
An \emph{$m$-string virtual tangle diagram} $T$ consists of 
$m$ intervals properly immersed in a $2$-disk $D^2$ 
whose singularities are only transverse double points.  
Each double point of $T$ is designated as 
either a classical crossing (with over-under information)  
or a virtual crossing (indicated by a small circle). 
In particular, $T$ is said to be \emph{classical} 
if it contains no virtual crossings. 
An \emph{arc} of $T$ is a connected component 
obtained by removing all under-crossings from the diagram. 

Let $R=\Z$ or $\Z/p\Z$, 
where $p\geq 3$ is odd. 
An \emph{$R$-coloring} of $T$ is a map  
\[C:\{\text{arcs of $T$}\}\rightarrow R\]
such that at each classical crossing of $T$, 
the relation 
\[C(x)+C(z)=2C(y)\] 
is satisfied, 
where $x$ and $z$ are the two under-arcs and 
$y$ is the over-arc. 
No requirement is imposed at virtual crossings. 
The value $C(x)$ assigned to an arc $x$ is called its \emph{color}, 
and the pair $(T,C)$ is called an \emph{$R$-colored diagram}. 
Let $E_1$ denote 
the $R$-coloring of $T$ 
in which all arcs are assigned the color $1$. 

The set of $R$-colorings of $T$ forms a $\Z$-module. 
For two $R$-colorings $C,C'$ of $T$ 
and $a,b\in R$, 
the linear combination $aC+bC'$ is the $R$-coloring 
defined by 
\[(aC+bC')(x)=aC(x)+bC'(x)\]
for any arc $x$ of $T$. 

For an $R$-colored virtual tangle diagram $(T,C)$, 
the colors appearing at the $2m$ boundary points 
in cyclic order along $\partial D^{2}$ form a vector 
\[
v=(a_1,\dots,a_{2m})\in R^{2m}.
\]
We call $v$ the \emph{boundary color vector} of $(T,C)$ 
and write $v=\partial(T,C)$. 
The index $i$ of $a_i$ is taken modulo $2m$. 
We remark that the choice of the starting endpoint 
for the vector $v$ is arbitrary 
and does not affect the arguments in this paper. 
A vector $v$ is said to be \emph{trivial} 
if $a_1=\dots=a_{2m}$, 
and \emph{nontrivial} otherwise. 
We denote by 
$e_1=(1,\dots,1)\in R^{2m}$ the trivial vector, 
whose entries are all equal to $1$. 

\subsection{$\Delta(v)$, $d(v)$, $k(v)$, and $M(v)$} 

For $v=(a_1,\dots,a_{2m})\in R^{2m}$, 
we define the alternating sum 
\[\Delta(v)=\sum_{i=1}^{2m}(-1)^{i-1}a_i\in R.\]

When $R=\Z$ and $v$ is nontrivial, 
let $d(v)$ denote the greatest common divisor 
of the differences between the entries of $v$: 
\[d(v)=\gcd\{a_i-a_j\mid 1\leq i\ne j\leq 2m\}.\] 
Since $v$ is nontrivial, 
$d(v)$ is a positive integer. 
It can be uniquely written in the form 
\[d(v)=2^k(2n+1)\] 
with $k,n\geq 0$. 
We define $k(v)=k$. 
Furthermore, 
let $M(v)$ denote the multiset consisting of 
the congruence classes of $a_1,\dots,a_{2m}$ modulo $2d(v)$: 
\[M(v)=\{a_1,\dots,a_{2m} \ ({\rm mod}~2d(v))\}.\]
For example, the vector $v=(1,7,-5,-11)\in\Z^4$ satisfies 
\[\Delta(v)=0,\ d(v)=6,\ k(v)=1,\text{ and }
M(v)=\{1,1,7,7~(\mathrm{mod}~12)\}.\]

By definition, 
for a nontrivial vector $v=(a_{1},\dots,a_{2m})\in\Z^{2m}$,  
we may write \[a_i=a_1+d(v)b_i\]
for some $b_i\in\Z$ $(1\leq i\leq 2m)$. 
Note that $b_1=0$. 
Let $w=(b_1,\dots,b_{2m})\in\Z^{2m}$ 
be the vector associated with $v$.

\begin{lemma}\label{lem21}
In the notation above, the following hold. 
\begin{enumerate}
\item
If there exists a $\Z$-colored virtual tangle diagram $(T,C)$ 
such that $\partial(T,C)=w$, 
then 
\[v=\partial(T,a_1E_1+d(v)C).\]
\item 
If $\Delta(v)$ is divisible by $2^{k(v)+1}$, 
then $\Delta(w)$ is even. 
\end{enumerate}
\end{lemma}

\begin{proof}
(i) 
Since $T$ admits the $\Z$-colorings 
$E_1$ and $C$, 
it also admits the $\Z$-coloring $a_1E_{1}+d(v)C$. 
Under this coloring, 
each arc $x$ of $T$ receives the color $a_1+d(v)C(x)$. 
Consequently the boundary color vector is given by 
\[(a_1+d(v)b_1,\dots,a_1+d(v)b_{2m})=a_1e_1+d(v)w=v.\]

(ii) From $a_i=a_1+d(v)b_i$ we obtain 
\[\Delta(v)=
\sum_{i=1}^{2m}(-1)^{i-1}(a_1+d(v)b_i)\\
=d(v)\sum_{i=1}^{2m}(-1)^{i-1}b_i=d(v)\Delta(w).\]
Since 
$d(v)=2^{k(v)}(2n+1)$ for some $n\geq 0$, 
the assumption that 
$\Delta(v)$ is divisible by $2^{k(v)+1}$ 
implies that $\Delta(w)$ is even. 
\end{proof}

\subsection{Hurwitz action} 

In this subsection, we briefly recall 
the Hurwitz action of the braid group on $\mathbb{Z}^{2m}$.
Let $B_{2m}$ denote the braid group on $2m$ strands  
with the standard generators $\sigma_{1},\dots,\sigma_{2m-1}$. 
The set $\Z^{2m}$ admits the \emph{Hurwitz action} of $B_{2m}$ 
defined by 
\[(a_1,\dots,a_i,a_{i+1},\dots,a_{2m})\cdot\sigma_i=
(a_1,\dots,a_{i+1},2a_{i+1}-a_i,\dots,a_{2m})\]
for $1\leq i\leq 2m-1$. 
If $v\cdot \beta=w$ 
for $v,w\in\Z^{2m}$ and $\beta\in B_{2m}$, 
then there exists a $\Z$-colored braid diagram 
$(\beta,C)$ whose top and bottom endpoints receive  
the vectors $v$ and $w$, respectively. 

The orbits of $\Z^{2m}$ 
under the Hurwitz action of $B_{2m}$ 
are described by certain normal forms and 
the three invariants $\Delta(v)$, $d(v)$, and $M(v)$ 
as follows. 

\begin{lemma}[{\cite[Proposition 5.4]{NNSW}}]\label{lem22}
Let $m\geq 2$. 
For any nontrivial vector $v\in\Z^{2m}$, 
there exist integers $a,\lambda$ 
and a braid $\beta\in B_{2m}$ such that 
\[v\cdot \beta
=(a,\dots,a,\underset{\text{$i$-th}}{\underline{a+\lambda d(v)}},
a+d(v),\dots,a+d(v))\]
for some $1<i<2m$. \qed
\end{lemma}

When $\Delta(v)=0$, 
this normal form simplifies as follows. 

\begin{lemma}[{\cite[Proposition~5.5]{NNSW}}]\label{lem23}
A (possibly trivial) 
vector $v\in\Z^{2m}$ satisfies $\Delta(v)=0$ 
if and only if 
there exist integers $a,b$ and a braid $\beta\in B_{2m}$ 
such that 
\[v\cdot\beta= (\underbrace{a,\dots,a}_{2r}, 
\underbrace{b,\dots,b}_{2s})\]
for some $r,s\geq 0$ with $r+s=m$. 
\qed
\end{lemma}

\begin{theorem}[{\cite[Theorem 1.1]{NNSW}}]\label{thm24}
For two nontrivial vectors $v,w\in\Z^{2m}$, 
there exists a braid $\beta\in B_{2m}$ 
such that $v\cdot \beta=w$ if and only if 
\[\Delta(v)=\Delta(w), \ d(v)=d(w), \text{ and }M(v)=M(w).\]
\qed
\end{theorem}

\section{Classical tangle diagrams}\label{sec3}

In this section, 
all tangle diagrams are assumed to be classical. 
Let $R=\Z$ or $\Z/p\Z$, where $p\geq 3$ is odd. 

\begin{lemma}[cf. {\cite[Lemma~1.5]{Prz}}]\label{lem31}
For any $R$-colored tangle diagram $(T,C)$, 
the boundary color vector $v=\partial(T,C)$ satisfies $\Delta(v)=0$. 
\end{lemma}

\begin{proof}
The proof follows by the argument of Przytycki used in~\cite{Prz}. 
Let $T'$ be the diagram obtained from $T$ 
by adding a large circle below $T$, 
as shown on the left of Figure~\ref{fig:Delta0}. 
The $R$-coloring $C$ of $T$ then 
extends naturally to an $R$-coloring of $T'$, 
since $T'$ is equivalent to a disjoint union of $T$ and a circle 
by a finite sequence of Reidemeister moves. 
See the right-hand side of the figure. 

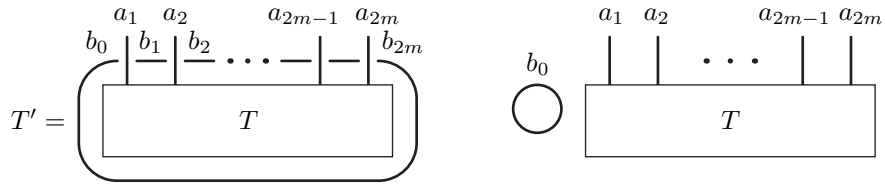
\begin{figure}[htb]
  \vspace{1em}
  \hspace{1em}
  \centering
    \begin{overpic}[]{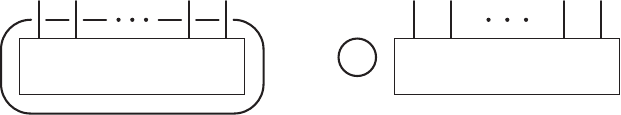}
      \put(-25,20){$T'=$}
      \put(60.5,20){$T$}
      \put(13.5,61){$a_{1}$}
      \put(32,61){$a_{2}$}
      \put(71.3,61){$a_{2m-1}$}
      \put(104,61){$a_{2m}$}
      \put(3,50){$b_{0}$}
      \put(23,50){$b_{1}$}
      \put(41.5,50){$b_{2}$}
      \put(114,50){$b_{2m}$}
      \put(241,20){$T$}
      \put(195,61){$a_{1}$}
      \put(212,61){$a_{2}$}
      \put(255,61){$a_{2m-1}$}
      \put(285,61){$a_{2m}$}
      \put(168,42){$b_{0}$}
    \end{overpic}
  \caption{Adding a circle to $T$}
  \label{fig:Delta0}
\end{figure}

Let $v=(a_1,\dots,a_{2m})$ be the boundary color. 
For $0\leq i\leq 2m$, 
let $b_i$ denote the color 
of the arc of the outer circle between 
the strands whose colors are $a_i$ and $a_{i+1}$, 
where the indices are taken modulo $2m$. 
Since $b_{i-1}+b_i=2a_i$, 
we obtain 
\[b_i=(-1)^i\biggl(b_0+2\sum_{j=1}^i(-1)^ja_j\biggr).\]
In particular, when $i=2m$ 
we have $b_{2m}=b_0-2\Delta(v)$.  
Since $b_{2m}=b_0$, 
it follows that $2\Delta(v)=0\in R$. 
As $R=\Z$ or $\Z/p\Z$ with $p\geq 3$ odd, 
the element $2$ is not a zero-divisor in $R$. 
Hence $\Delta(v)=0$. 
\end{proof}

\begin{proof}[Proof of {\rm Proposition~\ref{prop11}}.]
\underline{(i)$\Rightarrow$(ii).} 
This follows immediately from Lemma~\ref{lem31}. 

\underline{(ii)$\Rightarrow$(i).} 
Let $v=(a_{1},\dots,a_{2m})\in R^{2m}$ with $\Delta(v)=0$. 

We first consider the case $R=\Z$. 
By Lemma~\ref{lem23}, 
there exists a $\Z$-colored braid diagram $(\beta,C')$ such that 
the top endpoints receive the vector $v$ and 
the bottom endpoints receive 
\[w=(\underbrace{a,\dots,a}_{2r},\underbrace{b,\dots,b}_{2s})\] 
for some $r,s\geq 0$ with $r+s=m$. 

To construct a tangle diagram $T$, 
connect the $(2i-1)$-st and $2i$-th bottom endpoints of $\beta$ 
by simple arcs $(1\leq i\leq m)$. 
Each such arc joins endpoints with the same color. 
This produces a $\Z$-colored classical tangle diagram
$(T,C)$ satisfying $\partial(T,C)=v$. 
See Figure~\ref{fig:connect}. 
It is straightforward to verify that 
this diagram $T$ satisfies the property $\mathrm{P}_1$. 
Indeed, since $\beta$ is a braid diagram, 
$T$ can be reduced to a diagram with no crossings 
by repeated applications of the move shown in Figure~\ref{fig:interchange}(a). 

\begin{figure}[htb]
  \vspace{1em}
  \centering
    \begin{overpic}[]{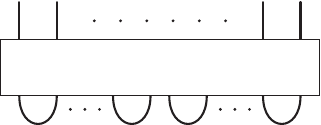}
      \put(63,24){$(\beta,C')$}
      \put(-43,24){$(T,C)=$}
      \put(4.5,65){$a_{1}$}
      \put(22.5,65){$a_{2}$}
      \put(108,65){$a_{2m-1}$}
      \put(139.5,65){$a_{2m}$}
      \put(15,-9){$a$}
      \put(60.5,-9){$a$}
      \put(88,-9){$b$}
      \put(133.5,-9){$b$}
    \end{overpic}
  \vspace{1em}
  \caption{Closing $(\beta,C')$ with $m$ arcs}
  \label{fig:connect}
\end{figure}
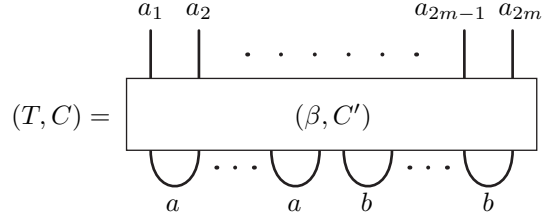

We next consider the case $R=\Z/p\Z$. 
Choose integers $b_i\in\Z$ ($1\leq i\leq 2m-1$)
such that $b_i=a_i\in\Z/p\Z$. 
Define 
\[b_{2m}=-\sum_{i=1}^{2m-1}(-1)^{i-1}b_i\]
and set $w=(b_1,\dots,b_{2m})\in\Z^{2m}$. 
Then $\Delta(w)=0\in\Z$. 
By the result for $R=\Z$ established above, 
there exists a $\Z$-colored classical tangle diagram $(T,C)$ 
such that $\partial(T,C)=w$. 
Since $\Delta(v)=0\in\Z/p\Z$, 
we have $b_{2m}=a_{2m}\in\Z/p\Z$. 
Let $\pi:\Z\rightarrow\Z/p\Z$ denote the natural projection. 
Then $(T,\pi\circ C)$ is a $\Z/p\Z$-colored tangle diagram 
with $\partial(T,\pi\circ C)=v$. 
Moreover, $T$ satisfies the property~$\mathrm{P}_1$. 
\end{proof}


\section{Virtual tangle diagrams}\label{sec4}

Throughout this section, unless stated otherwise, 
tangle diagrams are allowed to have virtual crossings.

\begin{lemma}\label{lem41}
Let $v=(a_1,\dots,a_{2m})\in\Z^{2m}$ be the boundary color vector 
of a $\Z$-colored virtual tangle diagram $(T,C)$. 
If $v$ is nontrivial, then the following hold. 
\begin{enumerate}
\item 
For any arc $x$ of $T$, we have 
\[C(x)\equiv a_1\pmod{2^{k(v)}}.\]
\item
For any two arcs $x$ and $y$ lying on the same string of $T$,  
we have 
\[C(x)\equiv C(y)\pmod{2^{k(v)+1}}.\]
\end{enumerate}
\end{lemma}

\begin{proof}
We prove by induction on $s$ that  
for each $0\leq s\leq k(v)$, 
the following hold: 
\begin{itemize}
\item[(i)$_s$] 
For any arc $x$ of $T$, we have 
\[C(x)\equiv a_1\pmod{2^s}.\]
\item[(ii)$_s$]
For any two arcs $x$ and $y$ lying on the same string of $T$,  
we have 
\[C(x)\equiv C(y)\pmod{2^{s+1}}.\] 
\end{itemize}

For $s=0$, (i)$_0$ is trivial 
as $C(x)\equiv a_1\pmod{1}$. 
To prove (ii)$_0$, consider any classical crossing 
at which $x,z$ are the two under-arcs and 
$y$ is the over-arc.
Since $C(x)+C(z)=2C(y)$,
we have \[C(x)-C(z)=2(C(y)-C(z)),\] 
and hence $C(x)\equiv C(z)\pmod{2}$. 
Since this holds at every classical crossing 
and virtual crossings do not affect colors, 
all arcs on the same string have the same parity. 
Thus (ii)$_0$ follows.

Assume that (i)$_s$ and (ii)$_s$ hold 
for some $0\leq s\leq k(v)-1$. 
To prove (i)$_{s+1}$, 
let $x$ be any arc of $T$. 
By (ii)$_s$, the color $C(x)$ is congruent 
modulo $2^{s+1}$ to the color $a_i$ 
at an endpoint of the same string. 
Since $s+1\leq k(v)$ and $a_i\equiv a_1\pmod{2^{k(v)}}$, 
we have $a_i\equiv a_1\pmod{2^{s+1}}$. 
Hence $C(x)\equiv a_1\pmod{2^{s+1}}$, 
proving (i)$_{s+1}$.

To prove (ii)$_{s+1}$, 
consider any classical crossing with arcs $x,y,z$ as above.  
From $C(x)-C(z)=2(C(y)-C(z))$ and 
$C(y)\equiv C(z)\pmod{2^{s+1}}$ by (i)$_{s+1}$, 
we obtain $C(x)\equiv C(z)\pmod{2^{s+2}}$. 
Therefore, the colors of any two arcs 
lying on the same string 
are congruent modulo $2^{s+2}$. 
Thus (ii)$_{s+1}$ follows. 

By induction, (i)$_{s}$ and (ii)$_{s}$ hold for all $0\leq s\leq k(v)$. 
Taking $s=k(v)$ gives the desired conclusions. 
\end{proof}

\begin{lemma}\label{lem42}
Let $v=(a_1,\dots,a_{2m})\in\Z^{2m}$ be the boundary color vector 
of a $\Z$-colored virtual tangle diagram $(T,C)$. 
If $v$ is nontrivial, 
then $\Delta(v)$ is divisible by $2^{k(v)+1}$. 
\end{lemma}

\begin{proof}
Write $a_i=a_1+2^{k(v)}b_i$ $(1\leq i\leq 2m)$ 
and set $w=(b_1,\dots,b_{2m})\in\Z^{2m}$. 

Suppose that $a_{i}$ and $a_{j}$ are the colors at 
the endpoints of the same string of $T$. 
By Lemma~\ref{lem41}(ii), we have 
$a_{i}\equiv a_{j}\pmod{2^{k(v)+1}}$, 
which implies that 
$b_i\equiv b_j\pmod{2}$. 
Since the integers $a_{1},\dots,a_{2m}$ 
are partitioned into  $m$ pairs by the $m$ strings of $T$, 
we obtain 
\[\Delta(w)\equiv\sum_{i=1}^{2m}b_i\equiv 0\pmod{2}.\]
Since $\Delta(v)=2^{k(v)}\Delta(w)$, 
it follows that $\Delta(v)$ is divisible by $2^{k(v)+1}$.
\end{proof}

The assumption in Lemma~\ref{lem42} 
can hold only when $m\geq 2$. 
Indeed, when $m=1$, 
we obtain the following. 

\begin{proposition}\label{prop43}
For any $\Z$-colored $1$-string virtual tangle diagram $(T,C)$, 
all arcs of $T$ have the same color. 
In particular, 
the boundary color vector 
$\partial(T,C)\in\Z^2$ is trivial. 
\end{proposition}

\begin{proof}
Let $\partial(T,C)=(a_{1},a_{2})\in\Z^{2}$. 
 As in the proof of Lemma~\ref{lem41}, 
we show by induction on $s$ that, 
for each $s\geq0$, 
any arc $x$ of $T$ satisfies 
\[C(x)\equiv a_1\pmod{2^s}.\]
Since $C(x)-a_1$ is divisible by $2^s$ for all $s\geq 0$, 
we have $C(x)=a_{1}$. 
Thus all arcs of $T$ have the same color $a_{1}$. 
\end{proof}

\begin{lemma}\label{lem44}
Suppose that two vectors 
$v,w\in\Z^{2m}$ satisfy 
$v\cdot\beta=w$ for some $\beta\in B_{2m}$. 
If there exists a $\Z$-colored 
virtual tangle diagram $(T,C)$ 
such that $\partial(T,C)=w$ and $T$
satisfies the properties $\mathrm{P}_2$ and $\mathrm{P}_3$, 
then the same holds for $v$. 
\end{lemma}

\begin{proof}
Let $(\beta,C')$ be the $\Z$-colored braid diagram 
associated with the Hurwitz action $v\cdot \beta=w$. 
Consider the $\Z$-colored tangle diagram $(T',C'')$ 
obtained by connecting $(\beta,C')$ to $(T,C)$, 
as shown in Figure~\ref{fig:Hurwitz-equiv}. 
Then, by construction, we have $\partial(T',C'')=v$. 

Since $T$ has at most one virtual crossing 
and $\beta$ has no virtual crossings, 
$T'$ satisfies the property $\mathrm{P}_2$. 
Moreover, since $T'$ is transformed into $T$ 
by a finite sequence of moves shown in Figure~\ref{fig:interchange}(a), 
$T'$ also satisfies the property $\mathrm{P}_3$. 
\end{proof}

\begin{figure}[htb]
  \centering
    \begin{overpic}[]{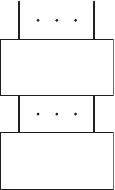}
      \put(60,80){$\leftarrow v$}
      \put(60,34){$\leftarrow w$}
      \put(14,57){$(\beta,C')$}
      \put(-51,40){$(T',C'')=$}
      \put(14,11){$(T,C)$}
    \end{overpic}
  \caption{Connecting $(\beta,C')$ to $(T,C)$}
  \label{fig:Hurwitz-equiv}
\end{figure}
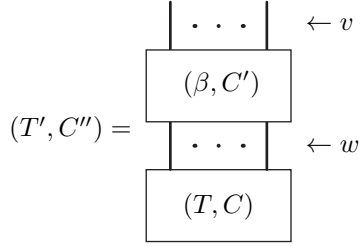

\begin{proof}[Proof of {\rm Theorem~\ref{thm12}}]
\underline{(i)$\Rightarrow$(ii).} 
This follows from Lemma~\ref{lem42}. 

\underline{(ii)$\Rightarrow$(i).}
(a) Any trivial vector $v=ae_{1}=(a,\dots,a)\in\Z^{2m}$ 
is realized by a $\Z$-colored 
classical tangle diagram $(T,aE_{1})$. 

(b) 
Since $v$ is nontrivial, 
Proposition~\ref{prop43} implies that $m\geq 2$. 
By Lemmas~\ref{lem22} and \ref{lem44}, 
we may assume that 
\[v=(a,\dots,a,\underset{\text{$i$-th}}{\underline{a+\lambda d(v)}},
a+d(v),\dots,a+d(v))\in\Z^{2m}\]
for some $a,\lambda\in\Z$ and $1<i<2m$. 
Furthermore, by Lemma~\ref{lem21}, 
the proof reduces to the case  
\[v=(0,\dots,0,\underset{\text{$i$-th}}{\underline{\lambda}},1,\dots,1)\]
under the assumption that $\Delta(v)$ is even. 

\underline{Case 1: $i$ is even.} 
Since $\Delta(v)=-\lambda$ is even, 
we may write $\lambda=2n$ for some $n\in\Z$. 
Consider the vectors $v',v''\in\Z^{2m}$ defined by 
\begin{align*}
v'&=(0,\dots,0,\underset{\text{$i$-th}}{\underline{n}},1-n,1,\dots,1) \text{ and}\\
v''&=(0,\dots,0,\underset{\text{$i$-th}}{\underline{n}},1,1-n,1,\dots,1). 
\end{align*}
By Theorem~\ref{thm24}, there exists $\beta\in B_{2m}$ 
such that $v\cdot \beta=v'$. 
Indeed, a direct computation shows that 
\begin{align*}
&\Delta(v)=\Delta(v')=-2n,\ 
d(v)=d(v')=1, \text{ and }\\
&M(v)=M(v')
=\{\underbrace{0,\dots,0}_{i},
\underbrace{1,\dots,1}_{2m-i} \ ({\rm mod}~2)\}.
\end{align*}
Let $(\beta,C')$ be the $\Z$-colored braid diagram 
associated with this Hurwitz action. 

On the other hand, since $\Delta(v'')=0$, 
Proposition~\ref{prop11} implies that  
there exists a $\Z$-colored classical tangle diagram $(T',C'')$ 
such that $\partial(T',C'')=v''$ and 
$T'$ satisfies the property $\mathrm{P}_1$. 

Connect $(\beta,C')$ and $(T',C'')$ by inserting a single virtual crossing 
between the arcs whose colors are $1-n$ and $1$, 
as shown in Figure~\ref{fig:construction}. 
This yields the desired $\Z$-colored virtual tangle diagram $(T,C)$ 
with $\partial(T,C)=v$, and  
$T$ satisfies $\mathrm{P}_2$ and $\mathrm{P}_3$.

\begin{figure}[htb]
  \vspace{1em}
  \centering
    \begin{overpic}[]{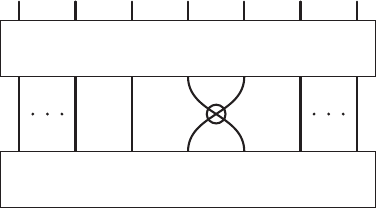}
      \put(6.7,104){$0$}
      \put(33.8,104){$0$}
      \put(50,104){$\lambda=2n$}
      \put(88,104){$1$}
      \put(115,104){$1$}
      \put(142,104){$1$}
      \put(169.3,104){$1$}
      \put(185,104){$\leftarrow v$}
      \put(1,53){$0$}
      \put(28,53){$0$}
      \put(54,53){$n$}
      \put(67,53){$1-n$}
      \put(119.5,53){$1$}
      \put(149,53){$1$}
      \put(176,53){$1$}
      \put(185,53){$\leftarrow v'$}
      \put(82,31){$1$}
      \put(119,31){$1-n$}
      \put(185,31){$\leftarrow v''$}
      \put(76,74.5){$(\beta,C')$}
      \put(-48,50){$(T,C)=$}
      \put(76,11){$(T',C'')$}
    \end{overpic}
  \caption{Connecting $(\beta,C')$ and $(T',C'')$ via a virtual crossing}
  \label{fig:construction}
\end{figure}
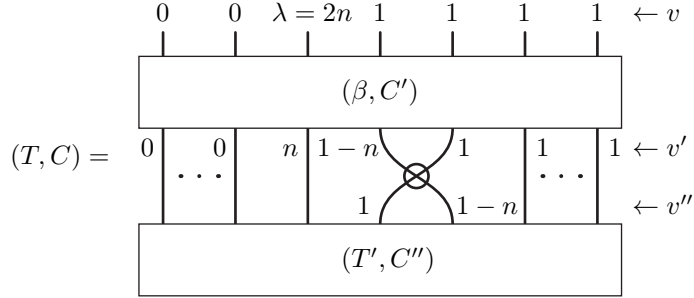

\underline{Case 2: $i$ is odd.} 
Since $\Delta(v)=\lambda-1$ is even, 
we may write $\lambda=2n+1$ for some $n\in\Z$. 
Consider the vectors $v',v''\in\Z^{2m}$ defined by 
\begin{align*}
v'&=(0,\dots,0,1-n,\underset{\text{$i$-th}}{\underline{n}},1,\dots,1) \text{ and}\\
v''&=(0,\dots,0,1-n,0,\underset{\text{$i$-th}}{\underline{n}},1,\dots,1). 
\end{align*}
Similarly to Case 1, 
since $v\cdot\beta=v'$ for some $\beta\in B_{2m}$ 
and $\Delta(v'')=0$, 
we can construct a $\Z$-colored tangle diagram $(T,C)$ 
with $\partial(T,C)=v$. 
Moreover, $T$ satisfies $\mathrm{P}_2$ and $\mathrm{P}_3$. 
\end{proof}

\begin{proof}[Proof of {\rm Corollary~\ref{cor13}}.] 
Let $v=(a_1,\dots,a_{2m})\in(\Z/p\Z)^{2m}$. 

For $m=1$, 
the required $\Z/p\Z$-colored diagram $(T,C)$ 
with a single virtual crossing 
is shown in Figure~\ref{fig:1-tangle}. 
Since $p\geq 3$ is odd, 
$\frac{p+1}{2}\geq2$ is an integer, 
and the coloring is well-defined.

\begin{figure}[htb]
  \centering
    \begin{overpic}[]{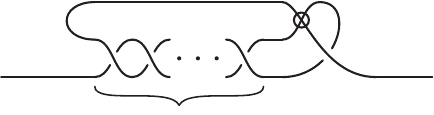}
      \put(20,43){$a_2$}
      \put(168.5,43){$a_2+\dfrac{p+1}{2}(a_2-a_1)$}
      \put(-12,16){$a_1$}
      \put(214,16){$a_2$}
      \put(140,7){$\nwarrow$}
      \put(141,-5){$a_2+\dfrac{p-1}{2}(a_2-a_1)$}
      \put(74.25,-12){$\dfrac{p+1}{2}$}
    \end{overpic}
  \vspace{1em}
  \caption{The $1$-string tangle diagram with a virtual crossing}
  \label{fig:1-tangle}
\end{figure}
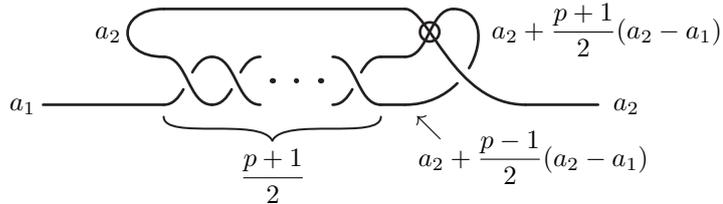

For $m\geq2$, 
choose integers $b_i\in\Z$ ($1\leq i\leq 2m$)
such that $b_i=a_i\in\Z/p\Z$. 
We may assume that the vector $w=(b_{1},\dots,b_{2m})\in\Z^{2m}$ 
contains both even and odd integers, 
by replacing $b_{i}$ with $b_{i}+p$ if necessary. 
Moreover, the number of odd integers among $b_{1},\dots,b_{2m}$ is assumed to be even. 

By assumption, $k(w)=0$ 
and $\Delta(w)$ is even. 
By Theorem~\ref{thm12}, 
there exists a $\Z$-colored 
virtual tangle diagram $(T,C)$ such that 
$\partial(T,C)=w$ and 
$T$ satisfies the properties P$_{2}$ and P$_{3}$. 
Let $\pi:\Z\rightarrow\Z/p\Z$ denote the natural projection. 
Then $(T,\pi\circ C)$ is the desired $\Z/p\Z$-colored tangle 
with $\partial(T,\pi\circ C)=v$.
\end{proof}

We conclude this paper with two remarks. 

\begin{remark}\label{rem45}
The following results provide 
an effective algorithm to determine whether 
a given vector $v=(a_1,\dots,a_{2m})\in\Z^{2m}$ 
can be realized as the boundary color vector 
of some $\Z$-colored virtual tangle diagram. 
The proof follows immediately from 
Theorem~\ref{thm12}. 

\begin{enumerate}
\item
Suppose that $a_1,\dots,a_{2m}$ are all odd. 
Then there exists a $\Z$-colored virtual tangle diagram 
$(T,C)$ such that $\partial (T,C)=v$ 
if and only if such a diagram exists 
for $(a_1-1,\dots,a_{2m}-1)$. 
\item
Suppose that $a_1,\dots,a_{2m}$ are all even. 
Then there exists a $\Z$-colored virtual tangle diagram 
$(T,C)$ such that $\partial (T,C)=v$ 
if and only if such a diagram exists 
for $(\frac{a_1}{2},\dots,\frac{a_{2m}}{2})$. 
\item
Suppose that $v$ contains both even and odd integers. 
Then there exists a $\Z$-colored virtual tangle diagram 
$(T,C)$ such that $\partial (T,C)=v$ 
if and only if 
the number of odd integers among $a_1,\dots,a_{2m}$ is even. 
\end{enumerate}
\end{remark}

\begin{remark}\label{rem46}
We consider the case where a tangle diagram 
is allowed to contain \emph{loop components}. 
For $R = \mathbb{Z}$ or $\mathbb{Z}/p\mathbb{Z}$, 
the realizability condition for $v\in R^{2m}$ 
by an $R$-colored classical tangle diagram (possibly with loops) 
remains $\Delta(v) = 0$.

In contrast, for ${\Z}$-colored virtual tangle diagrams, 
the realizability condition is different from Theorem~\ref{thm12}: 
For $v=(a_1,\dots,a_{2m})\in{\Z}^{2m}$, 
the following are equivalent. 
\begin{enumerate}
\item
There exists a ${\Z}$-colored virtual tangle diagram $(T,C)$ 
with possible loops such that $\partial(T,C)=v$. 
\item
The number of odd integers among $a_1,\dots,a_{2m}$ is even. 
\end{enumerate}

Indeed, if $v=\partial(T,C)$, 
then the colors at the endpoints of the same string of $T$
have the same parity, 
since (ii)$_0$ in the proof of Lemma~\ref{lem41} still holds 
with loop components. 
Conversely, if the number of odd entries in $v$ is even, 
 then $\Delta(v)$ is even. 
By defining 
\[a_{2m}'=a_{2m}+\Delta(v) \mbox{ and }
w=(a_1,\dots,a_{2m-1}, a_{2m}'),\] 
we have $\Delta(w)=0$. 
By Theorem~\ref{thm12}, 
there exists a ${\Z}$-colored classical tangle diagram 
$(T',C')$ with $\partial(T',C')=w$. 
The construction for the required ${\Z}$-colored diagram $(T,C)$ 
with one loop and one virtual crossing 
is shown in Figure~\ref{fig:loop}. 
\end{remark}

\begin{figure}[htb]
  \centering
    \begin{overpic}[]{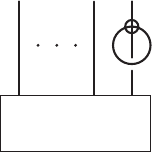}
      \put(4,78){$a_1$}
      \put(29,78){$a_{2m-1}$}
      \put(60,78){$a_{2m}$}
      \put(89,78){$\leftarrow v$}
      \put(76,51){$a_{2m}+\dfrac{\Delta(v)}{2}$}
      \put(67,33.5){$a_{2m}'$}
      \put(89,32){$\leftarrow w$}
      \put(19.5,11.5){$(T',C')$}
      \put(-40,49.2){$(T,C)=$}
    \end{overpic}
  \caption{The tangle diagram with a loop and a virtual crossing}
  \label{fig:loop}
\end{figure}
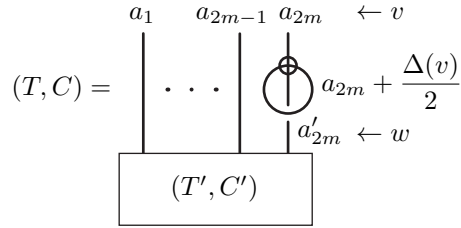

\begin{acknowledgements}
The authors would like to express their sincere gratitude 
to Professor Kouki Taniyama for his insightful suggestions. 
In particular, he provided an elementary proof of the fact that  
any knot admits only trivial ${\Z}$-colorings---consistent with 
the well-known property of odd determinants---by 
rewriting the coloring condition at each crossing as $a-c=2(b-c)$. 
This perspective, showing that $a-c$ is divisible by $2^s$ 
for any positive integer $s$ and thus $a=c$, 
was fundamental to the development of Lemma~\ref{lem41}. 
\end{acknowledgements}



\end{document}